\documentclass{amsart}
\usepackage{amssymb,latexsym}
\theoremstyle{plain}
\newtheorem{theorem}{Theorem}

\newtheorem{lemma}{Lemma}
\theoremstyle{definition}

\newtheorem{remark}{Remark}

\date{}

\begin{document}

\title[Network codes]
{Any network codes comes from an algebraic curve taking osculating spaces}
\author{Edoardo Ballico}
\address{Dept. of Mathematics, University of Trento, 38123 Povo (TN), Italy.}
\email{ballico@science.unitn.it}
\thanks{Partially supported by MIUR and GNSAGA (Indam)}

\maketitle

\begin{abstract}
In this note we prove that every network code over $\mathbb {F}_q$ may be realized taking some of the osculating spaces
of a smooth projective curve.
\end{abstract}

\section{Introduction}

In linear network coding informations go from one or more sources
to one of more targets in terms of a basis of a possible altered vector space
(\cite{lyc}, \cite{cwj03}, \cite{kk08}, \cite{tmr} and references therein). It is well-known that every linear code may be seen as a Goppa code over
a high genus curve (\cite{psv}). In this paper we prove that the same is true for network codes. We may also fix the curve $C$ with the only condition
that $\sharp (C(\mathbb {F}_q))$ is at least the number of subspaces of our network code. We need to say what is the equivalent
of a Goppa code in the set-up of network coding. We use the osculating codes introduced by J. Hansen in \cite{h1} and \cite{h2}. We do not need to restrict to constant dimension network coding.

Warning: in this paper we prove an existence theorem. We don't claim that it may be useful for the construction of network codes with good parameters.

Fix a finite field $\mathbb {F}_q$ and an integer $n>0$. Let $\mathcal {P}(\mathbb {F}_q^n)$
denote the set of all linear subspaces of the $\mathbb {F}_q$-vector space $\mathbb {F}_q^n$. A network code is just a non-empty subset $S\subseteq \mathcal {P}(\mathbb {F}_q^n)$.

For any vector space $W$ let $\mathbb {P}(W)$ denote the projective space of all lines of $W$ through $0$. In this note we prove the following result.

\begin{theorem}\label{a1}
Let $S\subseteq \mathcal {P}(\mathbb {F}_q^n)\setminus \{0\}$, $n\ge 3$, be any network code. Let $C$ be a smooth and geometrically connected
curve defined over $\mathbb {F}_q$ and such that $\sharp (C(\mathbb {F}_q)) \ge \sharp (S)$. Fix $A\subseteq C(\mathbb {F}_q)$ such that
$\sharp (A) =\sharp (S)$ and fix a bijection $u: A\to S$. For every $W\in S$ fix $Q_W\in \mathbb {P}(W) \subseteq \mathbb {P}(\mathbb {F}_q^n)$. Then there
exists a morphism $f: C \to \mathbb {P}^{n-1}$ defined over $\mathbb {F}_q$ such that $f(P) = Q_{u(P)}$ for every $P\in A$, $f$ has invertible differential
at each point of $A$ and $u(P)$ is the vector space whose associated projective space is the osculating space of dimension $\dim (u(P))-1$ to $(C,u)$ at $Q_{u(P)}$.
\end{theorem}

See section \ref{S2} for the definition of osculating spaces to a curve. We also gives another property of the map $f$, which is related to the definition of osculating space (see Remark \ref{b2}).

\begin{remark}\label{b1}
The map $u\vert A$ is injective if and only if $Q_a\ne Q_b$ for all $a\ne b$. By Hall's marriage theorem (\cite{b}, page 6) we may find a set $\{Q_s\}_{s\in S}$, with $Q_s\in U_s$ for all $s$
and $Q_a\ne Q_b$ for all $a\ne b$ if and only if for all $S'\subseteq S$ we have $\sharp (\cup _{s\in S'} U_s(\mathbb {F}_q)) \ge \sharp (S')$. This condition
is satisfied for any $S$ which may be interesting for network coding.
\end{remark}

We don't know if (in the case $n\ge 4$ and when the condition of Remark \ref{b1} is satisfied) we may take as $f$ an embedding. Another promising tool would be to use Poonen's Bertini's theorem over a finite field.
However, we certainly cannot prescribe the curve $C$ using Bertini's theorem. As far as we know no statement in the literature would give that $S$ comes
from an  embedding of a smooth curve. Two statements in \cite{p1}, i.e. \cite{p1}, Theorem 1.2 and Corollary 3.4, are very interesting, but not enough (as far as we see) to prove that any network code for which
the condition of Remark \ref{b1} is satisfied arises as osculating spaces of a smooth curve.

\section{The proof}\label{S2}
For any subset $E$ of a projective space, let $\langle E\rangle$ denote its linear span. 

Let $C$ be a smooth and connected curve defined over an algebraically closed field $\mathbb {K}$. We recall the notion of osculating spaces to a curve in the algebraic set-up (\cite{l}, \cite{hkt}, Ch. 7).
Fix a morphism $u: C \to \mathbb {P}^r$ defined over $\mathbb {K}$,
an integer $x>0$
and $P\in C(\mathbb {K})$. Assume $a:= \dim (\langle C(\mathbb {K})\rangle ) >x$. The completion $\widehat{\mathcal {O}}_{C,P}$ of the local ring $\mathcal {O}_{C,P}$
is the ring $\mathbb {K}[[t]]$ of power series in one variable over $\mathbb {K}$. We fix a homogenous system of coordinates $x_0,\dots ,x_r$ on $\mathbb {P}^r$ such that $u(P) =(1:0:\dots :0)$ and
$\langle C(\mathbb {K})\rangle = \{x_i=0$ for all $i>a\}$. Hence $u$ induces $r$ power series $f_1(t),\dots ,f_r(t)$ with $f_i(0) =0$ for all $i$, $f_i\ne 0$ for all $i\le a$ and $f_i\equiv 0$ for all $i>a$. We may find a linear change
of coordinates such that $f_1(t), \dots ,f_a(t)$ are power series with increasing order of zero. In this new coordinate system (call it $w_0,\dots ,w_r$, but notice that it depends
from $P$) the $x$-dimensional osculating space to $(C,u)$ at $P$ is the $x$-dimensional linear subspace $\{w_{x+1}= \cdots =w_r =0\}$. Write $f_x(t) =ct^e +g(t)$
with $c\ne 0$ and each term in the power series $g(t)$ of order $\ge e+1$. The definition of osculating space implies $e\ge x+1$. We say that $(C,u)$ has
an ordinary $x$-osculation at $P$ if $e=x+1$. In our construction we may often obtain ordinary $(d_P-1)$-osculation at each $P\in A$ (see Remark \ref{b2}).

\begin{lemma}\label{a5}
Fix linear subspaces $A_i\subset \mathbb {P}^r$, $B_i\subset \mathbb {P}^r$, $1\le i \le e$, such that $\dim (A_i)=\dim (B_i)$
for all $i$ and $\dim (A_1+\cdots +A_e) = \dim (B_1+\cdots +B_e) = \dim (A_1)+\cdots +\dim (A_e)+e-1$. Fix $O_i\in A_i$
and $O'_i\in B_i$.
Then there is an automorphism $h: \mathbb {P}^r\to  \mathbb {P}^r$ such that $h(A_i)=B_i$ and $h(O_i)=O'_i$ for all $i$. If all
$A_i,$, $B_i$, $O_i$ and $O'_i$ are defined over $\mathbb {F}_q$, then we may find $h$ defined over $\mathbb
{F}_q$.
\end{lemma}

\begin{proof}
The condition ``~$\dim (A_1+\cdots +A_e) = \dim (A_1)+\cdots +\dim (A_e)+e-1$~'' says that $A_i\ne A_j$ for all $i\ne j$
and that all these subspaces are linearly independent. The same condition holds for the linear subspaces $B_i$, $1\le i \le e$. Hence the existence
of $h$ such that $h(A_i)=B_i$ just
says that any linear independent subset of a finite-dimensional vector space $V$ may be completed to a basis and that
any two basis of $V$ differ by a linear automorphism of $V$. To get also the condition $h(O_i)=O'_i$ for all $i$ it
is sufficient to take $O_i$ (resp. $O'_i$) as one of the elements of a minimal subset of $A_i$ (resp. $B_i$) spanning
$A_i$ (resp. $B_i$).
\end{proof}

For any linear subspace $W\subsetneq \mathbb {P}^r$ let $\ell _W: \mathbb {P}^r\setminus A\to \mathbb {P}^{r-\dim (W)-1}$
denote
the linear projection from $W$.

\begin{lemma}\label{a6}
Let $U_i\subset \mathbb {P}^{n-1}$, $1\le i \le e$, be finitely many non-empty linear subspaces (we allow the case in which
$U_i \subseteq U_j$ or $U_i=U_j$ for some $i\ne j$). Let $U\subseteq \mathbb {P}^{n-1}$ be the linear span
of $\cup _{i=1}^{e} U_i$. Set $k:= \dim (U)$, $N:= \sum _{i=1}^{e} \dim
(U_i)+e-1$ and take any integer
$r\ge \max \{N,k\}$. Fix $Q_i\in U_i$, $1\le i \le e$. Let $A_i\subset \mathbb {P}^r$, $1\le i \le e$, be linear subspaces such
that
$\dim (A_i)=\dim (A_i)$ for all $i$ and $\dim (A_1+\cdots +A_e) =N$. Fix $O_i\in A_i$. Then there is a linear subspace
$W\subset U$ such that $W\cap A_i=\emptyset$ for all $i$, $\dim (W) =r-k-1$, $\ell _W(A_i)=U_i$ and $\ell _W(O_i)=Q_i$ for all
$i$. If all $U_i$,
$A_i$, $Q_i$ and $O_i$ are  defined over $\mathbb {F}_q$, then $W$ is defined over $\mathbb {F}_q$.
\end{lemma}

\begin{proof}
Let $E:= V_1\sqcup \cdots \sqcup V_e$ be the algebraic scheme with $e$ connected components $V_1,\dots ,V_e$ with
$V_i =U_i$ as an abstract scheme. The inclusion $\cup U_i\subset \mathbb {P}^{r-1}$ induces
a morphism $v: E \to U$ with $v(V_i)=U_i$ for all $i$ and $v\vert V_i: V_i\to U_i$ an isomorphism. Let $Q'_i\in V_i$ be the
only point such that $v(Q'_i)=Q_i$.
Set $\mathcal {L}:= v^\ast (\mathcal {O}_U(1))$. Since each $U_i$ is a linear subspace of $U$
and $V_i\cap V_j=\emptyset$ for all $i\ne j$, $\mathcal {L}$ is a very ample line bundle on $E$ and $h^0(E,\mathcal {L})=N+1$.
Let $w: E\hookrightarrow \mathbb {P}^N$ be the embedding of $E$ induced by the complete linear system
$\vert \mathcal {L}\vert$. Up to an automorphism of $\mathbb {P}^r$ we may assume that
$A_i = w(V_i)$ and $O_i =w(Q'_i)$ for all $i$ (Lemma \ref{a5}). Since $v$ is induced by a
subspace of
$\vert
\mathcal {L}\vert$, there is a linear subspace $W\subset \mathbb {P}^r$ such that $\dim (W)=r-k-1$ and $v = \ell _A\circ w$,
i.e. $U_i =\ell _A(A_i)$ for all $i$.
Since $\dim (U_i) =\dim (V_i)$, we have $W\cap U_i=\emptyset$ for all $i$. Since $w(Q'_i)=O_i$ and $v(Q'_i)=Q_i$, we
have $\ell _W(O_i)=Q_i$.
If all linear
spaces
$U_i$ and
$A_i$ are defined over a field
$K$, $Q_i\in U_i(K)$ and $O_i\in A_i(K)$, then the construction works over
$K$ and hence $W$ is defined
over $K$. 
\end{proof}

\begin{lemma}\label{a7}
Let $g\ge 0$ be the genus of $C$. For any $P\in A$ set $d_P:= \dim (u(P))+1$. Set $\delta := \sum _{P\in A} d_P$.
Fix any line bundle $L$ on $C$ defined over $\mathbb {F}_q$ and such that $\deg (L) \ge \max \{2g+1,2g-1 +\delta \}$.
Let $w: C \hookrightarrow \mathbb {P}^r$, $r:= \deg (L)-1$, be the embedding induced
by the complete linear system $\vert L\vert$. Each degree $d_P$ divisor $d_Pw(P)\subset w(C)$
spans a linear subspace $A_P$ defined over $\mathbb {F}_q$, $\dim (A_P) =d_P-1$ for all $P$
and $\langle \cup _{P\in A} A_P\rangle$ is a linear subspace of dimension $\delta$.  
\end{lemma}

\begin{proof}
$L$ is very ample, because $\deg (L)\ge 2g+1$. Fix an effective divisor $D\subset C$ such that $\deg (D)\le \delta$.
Since $\deg (L(-D)) \ge 2g-1$, we have $h^1(C,L(-D))=0$. Riemann-Roch gives
$h^0(C,L(-D)) = h^0(C,L) -\deg (D)$, i.e. $\dim (\langle w(D)\rangle )=\deg (D)-1$. Apply this observation
first to each divisor $d_PP$ and then to the divisor $\sum _{P\in A} d_PP$.
\end{proof}

\vspace{0.3cm}

\qquad {\emph {Proof of Theorem \ref{a1}.}} For each $P\in A$ set $\eta _P:= \max \{d_P,2\}$. Set $\eta := \sum _{P\in A} \eta _P$. Let
$L$ be a line bundle of degree $\ge \max \{2g+1,2g-1+\eta \}$ defined over $\mathbb {F}_q$.
Let $w: C \to \mathbb {P}^r$, $r:= \deg (L)-g$, denote the map associated to the complete linear system $\vert L\vert$.
Lemma \ref{a6} gives that $w$ is an embedding, that each linear space $A_P:= \langle d_Pw(P)\rangle$ has
dimension $d_P-1$ and that the suspaces $A_p$, $P\in A$, are linearly independent. Lemma \ref{a5} gives
the existence of a linear subspace $W\subset \mathbb {P}^r$ such that $\dim (W)=r-k-1$, $W\cap A_P=\emptyset$
for all $P$ and $u(P) = \ell _W(A_P)$ for all $P$. Since $\dim (A_p) = d_P-1$, $A_P$ is the
osculating space of dimension $d_P-1$ to $w(C)$ at $w(P)$.
Take $s =u(P)\in S$ such that $\dim (U_s) >0$. Since $A_P$ contains the tangent line to $w(C)$ at $w(P)$ and $W\cap A_P
=\emptyset$, $\ell
_W$ induces an embedding of $w(C)$ into $\mathbb {P}^{n-1}$ in a neighborhood of $w(P)$, i.e. we are computing the osculating
space with respect to a smooth branch of $u(C)$. This is not necessarily true at the points $s = u(P)$ such that $\dim
(U_s)=0$.
We do not need any smoothness to compute $u(P)$, but in the statement of Theorem \ref{a1} we claimed that we may find $u$
which is unramified at each point of $A$, i.e. that the differential of $u$ is invertible at each point of $A$. Set $S':=
\{s\in S:
\dim (U_s)>0\}$. For each
$s\in S\setminus S'$ choose any line $\widetilde{U}_s \subset \mathbb {P}^{n-1}$ defined over $\mathbb {F}_q$ and containing the
point
$U_s$. Set $S'':= S'\cup \bigcup _{s\in S\setminus S'} \widetilde{U}_s$ in which each line $\widetilde{U}_s$ has the point
$U_s$ as its prescribed $u(P) =Q_s$ (i.e. we use $\eta _P$ instead of $d_P$ and $\eta$ instead of $\delta$). If
$\widetilde{U}_s = U_a$
for some $s\in S\setminus S'$ and some $a\in S'$ or if $\widetilde{U}_a =\widetilde{U}_b$ for some $a\ne b$, then we count twice or more the linear space $\widetilde{U}_a$ (indeed, in Lemma \ref{a6} we allowed that some of the subspaces coincide). We chose
$Q_s$ as the point of $\widetilde{U}_s$. With this convention
the new map $u$ is unramified
at each point of $A$.\qed

\begin{remark}\label{b2}
Assume $\dim (U_s) \le n-3$ for all $s\in S$. For each $s\in S$ take a linear subspace $\overline{U}_s\subseteq \mathbb {F}_q^n$ such that
$\overline{U}_s\supset U_s$ and $\dim (\overline{U}_s) =\dim (U_s)+1$. Write $\overline{\delta }:= \delta +\sharp (S)$, i.e. write
$\overline{\delta}:= \sum _{s\in S} (\dim (\overline{U}_s)+1)$. Take $Q_s$ as the prescribed point of $\overline{U}_s$. Let $\overline{S}$ be the family of linear subspaces of
$\mathbb {F}_q^n$ counting them several times if we take $\overline{U}_a = \overline{U}_b$ for some $a\ne b$. Do the construction used to prove Theorem \ref{a1} using the family $\overline{S}$ with the points $\{Q_s\}_{s\in S}$. If $s = u(P)$, then $U_s$ is associated to the divisor $d_PP$, while $\overline{U}_s$
is associated to the divisor $(d_P+1)P$. Since $U_s\subsetneq \overline{U}_s$, we get that $(C,u)$ has an ordinary $(d_P-1)$-ramification at each $P\in A$.
\end{remark}

\providecommand{\bysame}{\leavevmode\hbox to3em{\hrulefill}\thinspace}

\end{document}